\documentclass[letterpaper, 10 pt, conference]{ieeeconf}  

\IEEEoverridecommandlockouts                              

\usepackage[utf8]{inputenc}
\usepackage{hyperref}
\usepackage{graphicx}
\usepackage{cite}
\usepackage{wrapfig}
\usepackage{array}
\usepackage{multirow}
\usepackage{amsmath, amscd, verbatim, amsfonts, amssymb}
\usepackage{graphics, graphicx, xcolor}
\usepackage{kotex}
\usepackage{algorithm}
\usepackage[noend]{algpseudocode}
\usepackage{caption}
\usepackage{subcaption}

\hypersetup{
    pdfborder = {0 0 0}
}

\newtheorem{Lemma}{\it Lemma}

\newtheorem{example}{\it Example}

\DeclareMathOperator*{\argmax}{argmax}
\DeclareMathOperator*{\argmin}{argmin}

\DeclareMathOperator{\diag}{\text{diag}}

\newcommand{\VEC}[1]{\text{vec}_h({#1})}
\newcommand{\invVEC}[1]{\text{vec}_h^{-1}({#1})}

\title{\LARGE \bf Distributed Kalman-filtering: Distributed optimization viewpoint}

\author{Kunhee Ryu and Juhoon Back${}^*$
\thanks{${}^*$Corresponding author}
\thanks{K. Ryu and J. Back are with School of Robotics, Kwangwoon University, Seoul, Republic of Korea {\tt\small \{ryuhhh, backhoon\}@kw.ac.kr}}
}

\begin{document}

\maketitle
\thispagestyle{empty}
\pagestyle{empty}

\begin{abstract}
We consider the Kalman-filtering problem with multiple sensors which are connected through a communication network. If all measurements are delivered to one place called fusion center and processed together, we call the process centralized Kalman-filtering (CKF). When there is no fusion center, each sensor can also solve the problem by using local measurements and exchanging information with its neighboring sensors, which is called distributed Kalman-filtering (DKF). Noting that CKF problem is a maximum likelihood estimation problem, which is a quadratic optimization problem, we reformulate DKF problem as a consensus optimization problem, resulting in that DKF problem can be solved by many existing distributed optimization algorithms. A new DKF algorithm employing the distributed dual ascent method is provided and its performance is evaluated through numerical experiments.
\end{abstract}

\section{Introduction}
It goes without saying that the Kalman-filter, an optimal state estimator for dynamic systems, has had a huge impact on various fields such as engineering, science, economics, etc. \cite{Welch1995, Bell1993TAC, Humpherys2010CSM, Thrun2005Book}. Basically, the filter predicts the expectation of the system state and its covariance based on the dynamic model and the statistical information on the model uncertainty or process noise, and then correct them using new measurement, sensor model, and the information on measurement noise. When multiple sensors possibly different types are available, we can just combine the sensor models to process the measurements altogether.

Thanks to the rapid development of sensor devices and communication technology, we are now able to monitor large scale systems or environments such as traffic network, plants, forest, sea, etc. In those systems, sensors are geometrically distributed, may have different types, and usually not synchronized. To process the measurements, the basic idea would be to deliver all the data to one place, usually called fusion center, and do the correction step as in the case of multiple sensors. This is called the centralized Kalman-filtering (CKF). As expected, CKF requires a powerful computing device to handle a large number of measurements and sensor models, is exposed to a single point of failure, and is difficult to scale up. In order to overcome these drawbacks, researchers developed the distributed Kalam-filtering (DKF) in which each sensor in the network solves the problem by using local measurements and communicating with its neighbors. Compared with CKF, DKF has advantageous in terms of the scalability, robustness to component loss, computational cost, and thus the literature on this topic is expanding rapidly \cite{Olfati2007CDC, Olfati2009CDC, Bai2011ACC, Carli2008SAC, Khan2008TSP, Kim2016CDC, Wu2016IFAC, WU2018Aut}. For more details on DKF, see the survey \cite{Mahmoud2013Survey} and references therein. 

Some relevant results are summarized as follows. In \cite{Olfati2007CDC}, the author proposed scalable distributed Kalman-Bucy filtering algorithms in which each node only communicates with its neighbors. An algorithm with average consensus filters using the internal models of signals being exchanged is proposed in \cite{Bai2011ACC}. It is noted that the algorithm works in a single-time scale. In the work \cite{Wu2016IFAC}, the authors proposed a continuous-time algorithm that makes each norm of all local error covariance matrices be bounded, thus overcomes a major drawback of \cite{Olfati2007CDC}. 
In \cite{Kim2016CDC}, an algorithm with a high gain coupling term in the error covariance matrix  is introduced and it is shown that the local error covariance matrix approximately converges to that of the steady-state centralized Kalman-filter. An in-depth discussion on distributed Kalman-filtering problem has been provided in \cite{Battistelli2015TAC, Battistelli2016Aut}, and the algorithms that exchange the measurements themselves, or exchange certain signals instead of the measurements are proposed, respectively.
 
Although each of the existing algorithms has own novel ideas and advantages, to the best of the authors' knowledge, we do not have a unified viewpoint for DKF problem. Motivated by this, it is the aim of this paper to provide a framework for the problem from the perspective of distributed optimization. 

We start by observing that the {\it{correction}} step of Kalman-filtering is basically an optimization problem \cite{Bell1993TAC, Humpherys2010CSM, Thrun2005Book}, and then formulate DKF problem as a consensus optimization problem, which provides a fresh look at the problem. This results in that DKF problem can be solved by many existing distributed optimization algorithms \cite{Boyd+2011FTML, Nedic+2009TAC, Nedic+2010TAC, Zhang2018CDC, Dorfler2017}, expecting various DKF algorithms to be  derived. As an instance, a new DKF algorithm employing the {\it{dual ascent method}} \cite{Dorfler2017}, one of the basic algorithms for distributed optimization problems, is provided in this paper.

This paper is organized as follows. In Section \ref{Sec:ProblemSetting}, we recall CKF problem from the optimization perspective, and connects DKF problem to a distributed optimization problem. 
A new DKF algorithm based on {\it{dual ascent method}} is proposed in Section \ref{Sec:DKF-DA},  and numerical experiments evaluating the proposed algorithm is conducted in Section \ref{Sec:NE}.
\medskip

\noindent{\bf Notation}: 
For matrices $A_1$, \dots, $A_n$, $\diag(A_1,\dots,A_n)$ denotes the block diagonal matrix composed of $A_1$ to $A_n$. For scalars $a_1$,\dots, $a_n$, $[a_1;\dots;a_n] := [a_1^\top,\dots,a_n^\top]^\top$, and $[A_1;\dots;A_n]$ with matrices $A_i$'s is defined similarly.  $1_n \in \mathbb{R}^n$ denotes the vector whose components are all 1, and $I_n$ is the identity matrix whose dimension is $n \times n$. The maximum and minimum eigenvalue of a matrix $A$ are denoted by $\sigma_{\max}(A)$ and $\sigma_{\min}(A)$, respectively. For a random variable $x$, $x \sim \mathsf{N}(\mu, \sigma^2)$ denotes $x$ is normally distributed with the mean $\mu$ and the variance $\sigma^2$, and $\mathbb{E}\{ x\}$ denotes the {\it{expected value}} of a random variable $x$, {\em{i.e.,}} $\mathbb{E}\{ x\} = \mu$. The half vectorization of a symmetric matrix $M \in \mathbb{R}^{n \times n}$ is denoted by $\VEC{M} \in \mathbb{R}^{n(n+1)/2}$, whose elements are filled in Column-major order. $i.e., \VEC{M} := [M_{1,1}; \dots; M_{1,n}; M_{2,2}; \dots; M_{2,n}; \dots;$ $ M_{n-1,n-1};M_{n-1,n};M_{n,n}]$ where $M_{i,j}$ is $i,j$ element of $M$, and $\invVEC{\cdot}$ denotes the inverse function of $\VEC{\cdot}$, $i.e., \invVEC{\VEC{M}} = M$. For a function $f(x, y): \mathbb{R}^{n}\times \mathbb{R}^m \rightarrow \mathbb{R}$, $\nabla_{x} f(x,y)$ denotes the gradient vector $\frac{\partial f(x,y)}{\partial x} = [\frac{\partial f(x,y)}{\partial x_1}; \dots;\frac{\partial f(x,y)}{\partial x_n}]$.


\medskip
\noindent{\bf Graph theory}: For a network consisting of $N$ nodes, the communication among nodes is modeled by a graph $\mathcal{G}$. Let ${\mathcal{A}} = [a_{ij}] \in {\mathbb{R}}^{N \times N}$ be an adjacency matrix associated to ${\mathcal{G}}$ where $a_{ij}$ is a weight of an edge between nodes $i$ and $j$. If node $i$ communicates to node $j$ then, $a_{ij} > 0$, or if not $a_{ij} = 0$. Assume there is not self edge, {\em i.e.}, $a_{ii} = 0$. The Laplacian matrix associated to the graph $\mathcal{G}$, denoted by $L$ is a $N \times N$ matrix such that $l_{ij, i \neq j} = -a_{ij}$, and $l_{ii} = \sum_{j=1}^N a_{ij}$. ${\mathcal{N}}_i$ is a set of nodes communicating with   node $i$, {\em i.e.}, ${\mathcal{N}}_i = \{j | a_{ij}>0 \}$.

\section{Distributed Kalman-filtering and Its Connection to Consensus Optimization}\label{Sec:ProblemSetting}
In this section, we recall CKF problem in terms of optimization, which is the maximum likelihood estimation\cite{Bell1993TAC}, and establish a connection between DFK and distributed optimization.

Consider a discrete-time linear system with $N$ sensors described by
\begin{subequations}\label{eq:System}
\begin{align}
x_{k+1} &= Fx_k + w_k\\
y_k &= H x_k + v_k = \begin{bmatrix} H_1 \\ H_2 \\ \vdots \\ H_N \end{bmatrix} x_k + \begin{bmatrix} v_{1,k} \\ v_{2,k} \\ \vdots \\ v_{N,k} \end{bmatrix}
\end{align}
\end{subequations}
where $x_{k} \in {\mathbb{R}}^{n}$ is the state vector of the dynamic system, $y_k := [y_{1,k};\dots;y_{N,k}] \in \mathbb{R}^{m}$ is the output vector, and $y_{i,k} \in \mathbb{R}^{m_i}$ is the output associated to sensor $i$. $m_i$'s satisfy $\sum^N_{i=1} m_i = m$. $F$ is the system matrix and $H$ is the output matrix consisting of $H_i \in \mathbb{R}^{m_i \times n}$ which is the output matrix associated to sensor $i$. $w_k\in {\mathbb{R}}^{n}$ with $w_k \sim \mathsf{N}(0, Q)$  is the process noise, $v_{i,k} \sim \mathsf{N}(0, R_i)$ is the measurement noise on  sensor $i$, and $v_k :=[v_{1,k};\dots;v_{N,k}] \in {\mathbb{R}}^{m}$ with $v_k \sim \mathsf{N}(0, \diag(R_1, \dots, R_N))$. Assume that the pair $(F, H)$ is observable, and each $v_{i,k}$ is uncorrelated to $v_{{j}, k}$ for $j \neq i$.

\subsection{Centralized Kalman-filtering problem from the optimization perspective}
If all the measurements from $N$ sensors are collected and processed altogether, the problem can be seen as the one with a imaginary sensor that measures $y_k$ with complete knowledge on $H$, thus called centralized  Kalman-filtering.
The filtering consists of two steps, {\it{prediction}} and {\it{correction}}. In the prediction step, the predicted estimate $\hat{x}_{k|k-1}$ and error covariance matrix $P_{k|k-1}$ are obtained based on the previous estimate, error covariance matrix, and the system dynamics. The update rules are given by
\begin{align*}
\hat{x}_{k|k-1} &= F \hat{x}_{k-1}\\
P_{k|k-1} &= {\mathbb{E}}\{ e_{k|k-1}e_{k|k-1}^\top \}\nonumber\\
&= F {\mathbb{E}}\{ e_{k-1}e_{k-1}^\top \}F^{\top} + {\mathbb{E}}\{ w_k w_k^\top \} \nonumber\\
&= F P_{k-1} F^\top + Q
\end{align*}
where $\hat{x}_{k-1}$ and $P_{k-1}$ are estimate and error covariance matrix in previous time, respectively, and $e_{k|k-1} := x_k - \hat{x}_{k|k-1}$, $e_{k} := x_k - \hat{x}_{k}$. Assume that $P_k$ is initialized as a positive definite matrix ($P_0>0$, usually set as $Q$). 

In the correction step, the predicted estimate and the error covariance matrix are updated based on the current measurements containing the measurement noise. The correction step can be regarded as a process to find the optimal parameter (estimate) from the predicted estimate $\hat{x}_{k|k-1}$, error covariance $P_{k|k-1}$, and the observation $y_k$. In fact, it is known that this step is an optimization problem ({\it{maximum likelihood estimation, MLE}}\cite{Bell1993TAC}) and we recall the details below. 

Let ${z}_k = [y_k;\hat{x}_{k|k-1}] \in \mathbb{R}^{m+n}$ and $\bar{H}_c = [H;I_n] \in \mathbb{R}^{(m+n) \times n}$. Then, ${z}_k \sim \mathsf{N}  ({\bar{H}_c} x_k, S_k)$ where $S_k =\text{diag}\{R, P_{k|k-1}\}$. For the random variable ${z}_k$, the likelihood function is given by
\begin{align*}
\mathfrak{L}(\xi_c) = \frac{1}{\sqrt{(2\pi)^{(m+n)}|S_k|}} e^{-\frac{1}{2}({z}_k-{\bar{H}_c} \xi_c)^{\top} S_k^{-1}({z}_k- {\bar{H}_c}\xi_c)}
\end{align*}
where the right-hand side is nothing but the probability density function of $z_k$ with the free variable $\xi_c \in \mathbb{R}^n$. 

Now, the maximum likelihood estimate $\hat{x}_{k}$ is defined as
\begin{align*} 
\hat{x}_{k} := \argmax_{\xi_c} (\mathfrak{L}(\xi_c)).
\end{align*}
Since $\mathfrak{L}(\xi_c)$ is a monotonically decreasing function with respect to $f_{\text{c}}(\xi_c) := \frac{1}{2}(z_k-\bar{H}_c \xi_c)^{\top}S^{-1} (z_k-\bar{H}_c\xi_c)$, $\hat{x}_k$ can also be obtained by
\begin{align}\label{eq:CKFUpdate}
\begin{split} 
\hat{x}_{k} &= \argmin_{\xi_c} (f_c(\xi_c))\\
&= \hat{x}_{k|k-1} + K (y_k - H\hat{x}_{k|k-1})
\end{split}
\end{align}
where $K = (H^\top R^{-1} H + P^{-1}_{k|k-1})^{-1} H^\top R^{-1}$. With the {\it{matrix inversion lemma}}, the Kalman-gain $K$ can be written as $K = P_{k|k-1}H^\top(H P_{k|k-1} H^\top + R)^{-1}$, which appears in  the standard Kalman-filtering.

On the other hand, by the definition of $P_{k} := {\mathbb{E}} \{ (\hat{x}_k - x_k)(\hat{x}_k - x_k)^\top \}$, the update rule of the error covariance matrix $P_{k}$ of CKF is given by
\begin{align}\label{eq:ECovUpdateCKF}
P_k &= (\bar{H}_c^\top S^{-1} \bar{H}_c)^{-1} = (H^\top R^{-1} H + P^{-1}_{k|k-1})^{-1}\\
&= P_{k|k-1} - (H^\top R^{-1} H + P^{-1}_{k|k-1})^{-1} H^{\top} R^{-1} H P_{k|k-1}. \nonumber
\end{align}
For more details, see \cite{Thrun2005Book, Bell1993TAC, Humpherys2010CSM}.

\subsection{Derivation of distributed Kalman-filtering problem}\label{Sec:DerivationDKF}
Now, we consider a sensor network which consists of $N$ sensors and suppose that each sensor runs an  estimator without the fusion center. Each estimator in the network tries to find the optimal estimate by processing the local measurement and exchanging information with its neighbors through communication network. The communication network among estimators is modeled by a graph $\mathcal{G}$ and the Laplacian matrix associated with $\mathcal{G}$ is denoted by $L \in \mathbb{R}^{N \times N}$. Under the setting (\ref{eq:System}), estimator $i$ measures only the local measurement $y_{k,i}$, and the parameters $H_i$ and $R_i$ are kept private to estimator $i$. It is noted that the pair $(F, H_i)$ is not necessarily observable. We assume that the graph is connected and undirected {\em{i.e.,}} $L = L^\top$, and $F$ and $Q$ are open to all estimators.

Similar to CKF, DKF has two steps, {\it{local prediction}} and {\it{distributed correction}}. In the local prediction step, each estimator predicts
\begin{align*}
\hat{x}_{i, k|k-1} &= F \hat{x}_{i, k-1}\\
P_{i, k|k-1} &= F P_{i, k-1} F^\top + Q.
\end{align*}
where $\hat{x}_{i, k|k-1}$ and $P_{i, k|k-1} $ are local estimates of $\hat{x}_{k|k-1}$ and $P_{k|k-1} $, respectively, that estimator $i$ holds.

In the distributed correction step, each estimator solves the maximum likelihood estimation in a distributed manner. The objective function of CKF $f_c (\xi)$ can be rewritten as
\begin{align*}
\sum_{i=1}^N f_i(\xi_c) = \sum_{i=1}^N \frac{1}{2} (\bar{z}_{i,k} - \bar{H}_i \xi_c)^\top \bar{S}_{i,k}^{-1} (\bar{z}_{i,k} - \bar{H}_i \xi_c)
\end{align*}
where $\bar{z}_{i,k} = [y_{i,k};\hat{x}_{i, k|k-1}]$, $\bar{H}_i = [H_i; I_n]$, $\bar{S}_{i,k} = \diag(R_{i}, N P_{i, k|k-1})$. 
 We assume that $\hat{x}_{i, k|k-1}=\hat{x}_{k|k-1}$ and $P_{i, k|k-1}=P_{k|k-1}$. This makes sense when the each sensor reached a consensus on $ \hat{x}_{i, k-1}$ and $P_{i, k-1}$ in the previous correction step.

Assuming that each estimator holds its own optimization variable $\xi_i \in \mathbb{R}^{n}$ for $\xi_c$,  DKF problem is written as the following consensus optimization problem.
\begin{subequations}\label{eq:DKFP}
\begin{align}
\text{minimize}& \quad \sum_{i=1}^N f_i(\xi_i)\label{eq:DKFP_obj}\\
\text{subject to}& \quad \xi_1 = \cdots = \xi_N.\label{eq:DKFP_const}
\end{align}
\end{subequations}
If there exists a distributed algorithm that finds a minimizer $(\xi_1^*, \dots, \xi_N^*)$, we say that the algorithm solves  DKF problem.

Since the kernel of Laplacian $L$ is $\text{span} \{1_N\}$, the constraints (\ref{eq:DKFP_const}) can be written with as $(L \otimes I_n) \xi = 0$ where $\xi = [\xi_1;\dots;\xi_N]$. To proceed, we define the Lagrangian to solve the problem (\ref{eq:DKFP}) as 
\begin{align}\label{eq:LagrangianDKFP}
{\mathcal{L}}(\xi, \lambda) &= \sum_{i=1}^N f_i(\xi_i) + \lambda^\top \bar{L} \xi
\end{align}
where $\lambda \in \mathbb{R}^{Nn}$ is the Lagrange multipliers (dual variable) associated with (\ref{eq:DKFP_const}) and $\bar{L} = (L \otimes I_n)$. We decompose the Lagrangian into local ones defined by
\begin{align}\label{eq:LagrangianDKFPLocal}
{\mathcal{L}}_i(\xi_i, \lambda_i) &= f_i(\xi_i) + \lambda_i^\top \sum_{j \in {\mathcal{N}}_i} a_{ij} (\xi_i - \xi_j).
\end{align}

For the Lagrangian (\ref{eq:LagrangianDKFP}), the partial derivatives over $\xi$ and $\lambda$ are given by
\begin{align*}
\nabla_{\xi} \mathcal{L}(\xi, \lambda) &= -\bar{H}^{\top} \bar{S}_k^{-1}(\bar{z}_k - \bar{H} \xi) + \bar{L} \lambda \\
\nabla_{\lambda} \mathcal{L}(\xi, \lambda) &= \bar{L} \xi,
\end{align*}
where $\bar{z}_k = [\bar{z}_{1,k};\dots;\bar{z}_{N,k}]$, $\bar{H}=\diag(\bar{H}_1, \dots, \bar{H}_N)$ and $\bar{S}_k = \diag(\bar{S}_{1,k}, \dots, \bar{S}_{N,k})$.
Then, the optimality condition for ($\xi^*$, $\lambda^*$) becomes the following saddle point equation (KKT conditions), namely
\begin{align}\label{eq:SPE}
\begin{bmatrix} -\bar{H}^\top \bar{S}_k^{-1} \bar{H} & -\bar{L} \\ \bar{L} & {{0}} \end{bmatrix}
\begin{bmatrix} \xi^* \\ \lambda^* \end{bmatrix} = \begin{bmatrix} -\bar{H}^\top \bar{S}_k^{-1} \bar{z}_k\\ 0 \end{bmatrix}
\end{align}
where $\xi^* := [\xi_1^*; \dots; \xi_N^*]$ and $\lambda^* := [\lambda_1^*; \dots; \lambda_N^*]$.

\begin{Lemma}\label{Lemma:DKFConvergence}
The solutions to DKF problem are parameterized as $(\xi^*, \lambda^*)=((1_N\otimes I_n) \xi^\dagger, (1_N\otimes I_n) \tilde \lambda +  \bar \lambda)$ where $\xi^\dagger\in \mathbb R^n$ and $\bar \lambda \in \mathbb R^{Nn}$ are unique vectors and $\tilde \lambda\in \mathbb R^n$ is an arbitrary  vector. If $(\xi^*, \lambda^*)$ is an optimal solution to DKF problem, then $\xi_i^*$ is the optimal solution to CFK problem. \hfill $\diamond$
\end{Lemma}
\begin{proof}
By multiplying $1_N^\top \otimes I_n$ to the dual feasibility equation in (\ref{eq:SPE}), one can obtain 
\begin{align}\label{eq:proofLemma1}
(1_N^\top \otimes I_n) \bar{H}^\top \bar{S}_k^{-1} \bar{H} \xi^* = (1_N^\top \otimes I_n) \bar{H}^\top \bar{S}_k^{-1} \bar{z}_k.
\end{align}

The primal feasibility equation in (\ref{eq:SPE}) implies that $\xi^*=(1_N \otimes I_n) \xi^\dagger$, hence (\ref{eq:proofLemma1}) becomes
\begin{align*}
(1_N^\top \otimes I_n) \bar{H}^\top \bar{S}_k^{-1} \bar{H} (1_N \otimes I_n)  \xi^\dagger = (1_N^\top \otimes I_n) \bar{H}^\top \bar{S}_k^{-1} \bar{z}_k.
\end{align*}
From $\bar{H}^\top \bar{S}_k^{-1} \bar{H} = \diag(\frac{1}{N}P^{-1}_{k|k-1} + H_1^\top R_1^{-1} H_1, \dots,$ $\frac{1}{N}P^{-1}_{k|k-1} +  H_N^\top R_N^{-1} H_N)$, one has
\begin{align*}
\Big (P^{-1}_{k|k-1} &+ \sum_{i=1}^N H_i^{\top} R_i^{-1} H_i\Big) \xi^\dagger \\
&= P^{-1}_{k|k-1} \hat{x}_{k|k-1} + \sum_{i=1}^N H_i^\top R_i^{-1} y_{i,k}.
\end{align*}

Since  $\sum_{i=1}^N H_i^\top R_i^{-1} y_{i,k} = H^\top R^{-1} y_k$ and $\sum_{i=1}^N H_i^\top R_i^{-1} H_i = H^\top R^{-1} H$,  it follows that 
\begin{align}
\xi^\dagger = \hat{x}_{k|k-1} + K_k (y_k - H\hat{x}_{k|k-1})
\end{align}
where $K_k = (P^{-1}_{k|k-1} + H^\top R^{-1} H)^{-1} H^\top R^{-1}$ and by the matrix inversion lemma, we have $K_k = P_{k|k-1} H^\top (R + H P_{k|k-1} H^\top)^{-1}$. From the fact that the right-hand side of above equation is the same with the update rule (\ref{eq:CKFUpdate}) of CKF, it follows that $\xi_i^{*}=\xi^\dagger$ is the optimal estimate of CKF $\hat{x}_k$. 

On the other hand, one can observe that the optimal dual variable $\lambda^*$ is not unique since the dual feasibility equation 
\begin{align}\label{eq:lambda_star}
(L \otimes I_n) \lambda^* = \bar{H}^\top \bar{S}_k^{-1} (\bar{z}_k - \bar{H} (1_N \otimes I_n) \xi^\dagger)
\end{align}
is singular. To find $\eta^*$, consider the orthonormal matrix $U = [U_1 ~ \bar{U}]$ such that $LU = U \Lambda$ where $U_1 = \frac{1}{\sqrt{N}} 1_N$, $\bar{U}$ consists of the eigenvectors associated with the non-zero eigenvalues of $L$, denoted by $\sigma_2, $\dots$, \sigma_N$, and $\Lambda = \diag({0, \sigma_2, \dots, \sigma_N})$.  Left multiplying $U^\top\otimes I_n$ to the equation \eqref{eq:lambda_star} yields
\begin{align*}
\left( \begin{bmatrix}
0 & \\
 & \bar{\Lambda}
\end{bmatrix} \otimes I_n \right) \left( \begin{bmatrix} U_1^\top \\ \bar{U}^\top \end{bmatrix} \otimes I_n \right) \lambda^* = \left( \begin{bmatrix} U_1^\top \\ \bar{U}^\top \end{bmatrix} \otimes I_n \right) b
\end{align*}
where $\bar{\Lambda} = \diag(\sigma_2, \dots, \sigma_N)$ and $b = \bar{H}^\top \bar{S}_k^{-1} (\bar{z}_k - \bar{H} (1_N \otimes I_n) \xi^\dagger)$. Hence, the optimal dual variable $\lambda^*$ becomes
$\lambda^* = (U\otimes I_n)\begin{bmatrix} \tilde{\lambda}^* ; 
(\bar{\Lambda}^{-1} \bar{U}^\top \otimes I_n)b
\end{bmatrix}$ where $\tilde{\lambda}^* \in \mathbb{R}^{n}$ is an arbitrary vector. This completes the proof.
\end{proof}
\subsection{Information form of DKF problem}
It is well known that the dual of the Kalman-filter is the {\it{Information filter}} which uses the {\it{canonical parameterization}} to represent the normal (Gaussian) distribution \cite{Thrun2005Book}. With the canonical parameterization, DKF problem (\ref{eq:DKFP}) can also be written in information form.

Let $\eta_i = (H_i^\top R^{-1}_i H_i + \frac{1}{N} \Omega_{i, k|k-1})\xi_i$, $\Omega_{i, k|k-1} = P_{i, k|k-1}^{-1}$ and $\tau_{i, k|k-1} = P_{i, k|k-1}^{-1}\hat{x}_{i, k|k-1}$ which are the local decision variable for the information vector of the estimator $i$, the locally predicted information matrix and information vector, respectively. With these transformations, we rewrite the problem (\ref{eq:DKFP}) as
\begin{subequations}\label{eq:DIFP}
\begin{align}
\text{minimize}& \quad \sum_{i=1}^N h_i(\eta_i)\label{eq:DIFP_obj}\\
\text{subject to}& \quad \eta_1 = \cdots = \eta_N\label{eq:DIFP_const}
\end{align}
\end{subequations}
where 
\begin{align*}
h_i(\eta_i) = \frac{1}{2} &\Big( \eta_i^\top \Phi_i^{-1} \eta_i - \eta_i^\top \Phi_i^{-1} (H_i^\top R_i^{-1} y_i + \frac{1}{N}\tau_{i, k|k-1}) \\
& \hspace{0.9cm} + y_i^\top R_i^{-1}y_i + \frac{1}{N}\tau_{i, k|k-1}^\top \Omega_{i, k|k-1}^{-1} \tau_{i, k|k-1} \Big)
\end{align*}
and $\Phi_i = H_i^\top R^{-1}_i H_i + \frac{1}{N} \Omega_{i, k|k-1}$. For the distributed problem (\ref{eq:DIFP}), the Lagrangian is given by
\begin{align*}
{\mathcal{L}}_{\eta}(\eta, \lambda) &= \sum_{i=1}^N h_i(\eta_i) + \nu^\top \bar{L} \eta
\end{align*}
where $\eta := [\eta_1;\dots;\eta_N]$ and $\nu$ is the Lagrange multipliers. The associated saddle point equation becomes
\begin{align*}
\begin{bmatrix} -(\bar{H}^\top \tilde{S}_k^{-1} \bar{H})^{-1} & -\bar{L} \\ \bar{L} & {0} \end{bmatrix}
\begin{bmatrix} \eta^* \\ \nu^* \end{bmatrix} = \begin{bmatrix} -\bar{H}^\top \tilde{S}_k^{-1} \tilde{z}_k \\ 0 \end{bmatrix}
\end{align*}
where $\tilde{z}_k = [\tilde{z}_{1,k};\dots;\tilde{z}_{N,k}]$, $\tilde{S}_k = \diag(\tilde{S}_{1,k}, \dots, \tilde{S}_{N,k})$, $\tilde{z}_{i,k} = [y_{i, k}; \tau_{i,k|k-1}]$ and $\tilde{S}_{i,k} = \diag(R_i, N\Omega_{i, k|k-1}^{-1})$.

\subsection{Interpretations of existing DKF algorithm from the optimization perspective} \label{Ssec:DiscussionOnOthers}
One of the recent DKF algorithms, {\it{Consensus on Information}} (CI) \cite{Battistelli2015TAC, Battistelli2016Aut}  can be interpreted in the provided framework. CI consists of three steps, {\it{prediction}}, {\it{local correction}}, and {\it{consensus}}. In the prediction step, each estimator predicts the estimate based on the system dynamics and previous estimate similar to the standard information filter algorithm. Each estimator also updates the estimate with local measurements and output matrix in the local correction step. After that, the estimators find the agreed estimate by averaging the local estimates in the consensus step.

In the provided framework, CI can be viewed as the algorithm which solves the problem \eqref{eq:DIFP} through the two steps, the local correction step and the consensus step. In the former step, each of estimators finds the local minimizer (estimate) of the local objective function $h_i(\cdot)$. Since the partial derivative of $h_i(\eta_i)$ becomes
\begin{align*}
\nabla_{\eta_i} h_i(\eta_i) = \Phi_i^{-1}\eta_i - \Phi_i^{-1} (H_i^\top R_i^{-1} y_i + \frac{1}{N}\tau_{i, k|k-1})
\end{align*}
and the local minimizer $\eta_i^*$ can be obtained by $\eta^*_i = H_i^\top R_i^{-1} y_i + \frac{1}{N}\tau_{i, k|k-1}$, which is the local update rule of CI\footnote{In the CI, the scalar $\frac{1}{N}$ is neglected \cite{Battistelli2015TAC}.}. The local minimizer, however, can be different among estimators, since it minimizes only the local objective function $h_i(\cdot)$, which violates the constraint (\ref{eq:DIFP_const}).

The consensus step of CI performs a role to find an agreed (average) value of the local estimates, using the doubly stochastic matrix, and the results of the consensus step satisfy the constraint (\ref{eq:DIFP_const}). The agreed estimate, however, may not be the global minimizer of (\ref{eq:DIFP}), which means that the consensus step cannot guarantee the convergence of the estimates to that of CKF.

\section{A Solution to DKF Problem}\label{Sec:DKF-DA}
One can observe that (\ref{eq:LagrangianDKFP}) is strictly convex, differentiable, and the local objective function $f_i(\cdot)$ is a quadratic function, hence {\it{strong duality}} holds. In addition, from the fact $\bar{H}^\top \bar{S}_k^{-1} \bar{H}$ is a nonsingular and block diagonal matrix, the optimal conditions (\ref{eq:SPE}) are already in a distributed form. This implies that the minimizer $\xi^*$ can be obtained in a distributed manner as long as $\lambda^*$ is given, {\em{i.e.,}} $\xi_i^* = (\bar{H}_i^\top \bar{S}_{i,k}^{-1} \bar{H}_i)^{-1} (\bar{H}_i^\top \bar{S}_{i,k}^{-1} \bar{z}_{i,k} - \sum_{j \in \mathcal{N}_i} a_{ij} (\lambda^*_{i} - \lambda^*_{j}))$.

Based on the above discussion, we see that one possible algorithm solving (\ref{eq:DKFP}), guaranteeing the asymptotic convergence to the global minimizer $\xi^*$, is the dual ascent method \cite{Boyd+2011FTML, Dorfler2017} which is given by 
\begin{subequations}\label{eq:UpdateDAGlobal}
\begin{align}
\xi_{l+1} &= (\bar{H}^\top \bar{S}_k^{-1} \bar{H})^{-1} (\bar{H}^{\top} \bar{S}_k^{-1} \bar{z}_k - \bar{L} \lambda_l) \label{eq:PrimalUpdateGlobal}\\
\lambda_{l+1} &= \lambda_l + \alpha_\lambda \bar{L} \xi_{l+1}\label{eq:DualUpdateGlobal}
\end{align}
\end{subequations}
where $\alpha_\lambda > 0$ is a step size. The update rule \eqref{eq:UpdateDAGlobal} can be written locally as 
\begin{subequations}\label{eq:DKF_DA}
\begin{align}
\xi_{i,l+1} &= \hat{x}_{i, k|k-1} + K_{i,k} (y_{i,k} - H_i \hat{x}_{i, k|k-1}) - \psi_{i,l}\\
\lambda_{i,l+1} &= \lambda_{i,l} + \alpha_\lambda \sum_{j \in \mathcal{N}_i} a_{ij} (\xi_{i,l+1} - \xi_{j,l+1}).
\end{align} 
\end{subequations}
where $K_{i,k} = (H_i^{\top} R_i^{{-1}} H_i + \frac{1}{N} P_{i, k|k-1}^{{-1}})^{-1} H_i^{\top} R_i^{-1}$, $\psi_{i,l} = (H_i^\top R_i^{-1} H_i + \frac{1}{N} P_{i, k|k-1}^{-1})^{-1} \sum_{j \in \mathcal{N}_i} a_{ij} (\lambda_{i,l} - \lambda_{j,l})$, and $l$ is the iteration index to find the minimizer. 

Regarding the convergence of the update rule \eqref{eq:DKF_DA}, we have the following result.
\begin{Lemma}\label{lem:convergence}
Assume that the network $\mathcal{G}$ is undirected and connected. Then, the sequence $\{ \xi_{i,l} \}$ generated by the dual ascent method (\ref{eq:DKF_DA}) converges to $\hat{x}_k$ of CKF problem (\ref{eq:CKFUpdate}), as $l$ goes to infinity, provided that the step size $\alpha_\lambda > 0$ is chosen such that 
\begin{equation}\label{eq:bound_step_size}
\alpha_\lambda < \frac{2}{\sigma_N^2 \max_{i} \{ \| (\bar{H}_i^\top S_{i,k}^{-1} \bar{H}_i)^{-1}\| \}}
\end{equation}
where $\sigma_N$ is the maximum eigenvalue of $L$. Moreover, the sequence $\{\lambda_{i,l} \}$ converges to a vector which is uniquely determined by the initial conditions of  $\lambda_i$'s.  \hfill $\diamond$
\end{Lemma}
\begin{proof}
Substituting the dual feasibility equation to the primal feasibility equation of (\ref{eq:SPE}) yields
\begin{align}\label{eq:DualOpt}
\bar{L} (\bar{H}^\top \bar{S}_k^{-1} \bar{H})^{-1} \bar{L} \lambda^* = \bar{L} (\bar{H}^\top \bar{S}_k^{-1} \bar{H})^{-1} \bar{H}^\top \bar{S}_k^{-1} \bar{z}_k.
\end{align} 
Now let $e^\lambda_{l} = \lambda_{l} - \lambda^*$. Then, one obtains 
\begin{align*}
e^\lambda_{l+1} &= \lambda_l + \alpha_\lambda \bar{L} \xi_{l+1} - \lambda^*\\
&= \lambda_l + \alpha_\lambda \bar{L}(\bar{H}^\top \bar{S}_k^{-1} \bar{H})^{-1}(\bar{H}^\top \bar{S}_k^{-1} \bar{z}_k - \bar{L}\lambda_l) - \lambda^*.
\end{align*}
From the identity (\ref{eq:DualOpt}), we have
\begin{align}\label{eq:DualErrorDynamics}
\begin{split}
e^\lambda_{l+1} &= (I - \alpha_\lambda \bar{L} (\bar{H}^\top \bar{S}_k^{-1} \bar{H})^{-1} \bar{L}) e^{\lambda}_l\\
&:= (I - \alpha_\lambda \tilde{A}_\lambda) e^{\lambda}_l.
\end{split}
\end{align}
Here, $\tilde{A}_\lambda$ is a symmetric positive semi-definite matrix which has $n$ simple zero eigenvalues, and it holds that
$I - \alpha_\lambda \sigma_{\max}(\tilde{A}_\lambda)I \leq I - \alpha_\lambda \tilde{A}_\lambda \leq I - \alpha_\lambda \sigma_{\min}(\tilde{A}_\lambda)I$. 
Since $\sigma_{\min}(\tilde{A}_\lambda)$ is zero, it follows that if $\alpha_\lambda > 0$ is chosen such that $\alpha_\lambda \sigma_{\max}(\tilde A_{\lambda}) < 2$, all eigenvalues of $I - \alpha_\lambda \tilde{A}_\lambda$, except $1$, are located inside the unit circle. The bound \eqref{eq:bound_step_size} ensures this.

Regarding the convergence of $\lambda_l$, we proceed as follows. With the orthonormal matrix $U$ used in {\it{Lemma}} \ref{Lemma:DKFConvergence}, $\tilde{A}_\lambda$ can be written as 
\begin{align*}
\tilde{A}_\lambda &= (U \Lambda U^\top \otimes I_n) (\bar{H}^\top \bar{S}_k^{-1} \bar{H})^{-1} (U \Lambda U \otimes I_n) \\
&= (U \otimes I_n) \diag(0_n, M_{\text{sub}}) (U^\top \otimes I_n)
\end{align*}
where ${M}_{\text{sub}} \in \mathbb{R}^{(N-1)n \times (N-1)n} $ is a submatrix with the first $n$ rows and first $n$ columns removed. In the new coordinates $\bar{e}^\lambda_l$, defined by $\bar{e}^\lambda_l=  (U^\top \otimes I_n )e^\lambda_l $, the error dynamics of the dual variable can be expressed as
\begin{align*}
\bar{e}^\lambda_{l+1} &= \diag(I, I - \alpha_\lambda {M}_{\text{sub}}) \bar{e}^\lambda_{l}.
\end{align*}
From this equation, we know that the first $n$ components of $\bar{e}^\lambda_{l}$, denoted by $\tilde{e}^\lambda_{l}$, remains the same for any $l$, {\em i.e.}, $\tilde e_l^\lambda = \tilde e_0^\lambda$, $\forall l\ge 0$, meaning that
$(U_1^\top \otimes I_n) e^\lambda_l =\tilde e_0^\lambda, \ \forall l\ge 0$, 
which means that $\tilde e_0^\lambda= (U_1^\top \otimes I_n) e^\lambda_0$.
Moreover, with  $\alpha_\lambda$ chosen as \eqref{eq:bound_step_size}, which guarantees that the matrix $I - \alpha_\lambda {M}_{\text{sub}}$ has all its eigenvalues except 1 inside the unit circle, we have
$\lim_{l \rightarrow \infty} \bar{e}^\lambda_{l} = \begin{bmatrix} \tilde{e}^\lambda_{0}; 0 \end{bmatrix}$, 
from which it follows that 
\begin{equation}\label{eq:Limit_e_lambda}
\lim_{l \rightarrow \infty} e^\lambda_l =(U\otimes I_n)  \begin{bmatrix} \tilde{e}^\lambda_{0}; 0 \end{bmatrix}= (U_1\otimes I_n)(U_1^\top \otimes I_n) e^\lambda_0.
\end{equation}
Recalling that $e^\lambda_l := \lambda_l - \lambda^*$, we have from \eqref{eq:Limit_e_lambda}
\begin{align*}
\lim_{l \rightarrow \infty} \lambda_l &= \lambda^* + (U_1 U_1^\top \otimes I_n) (\lambda_0 - \lambda^*).
\end{align*}
Applying $\lambda^* = (U_1 \otimes I_n) \tilde{\lambda}^* + (\bar U \bar{\Lambda}^{-1} \bar{U}^\top \otimes I_n)b$ (for $\tilde \lambda^*$ and $b$, see the proof of {\em Lemma} \ref{Lemma:DKFConvergence}), we have
\begin{align*}
\lim_{l \rightarrow \infty} \lambda_l &=(\bar U \bar{\Lambda}^{-1} \bar{U}^\top \otimes I_n)b + (1_N \otimes I_n) \text{avg}(\lambda_{i, 0})
\end{align*}
where $\text{avg}(\lambda_{i, 0}) = \frac{1}{N} \sum_{i=1}^N \lambda_{i, 0}$, and this completes the proof.
\end{proof}

Now, we derive an update rule of the error covariance matrix. With the information matrix $\Omega_k := P_k^{-1}$, the error covariance update rule (\ref{eq:ECovUpdateCKF}) can be written as
\begin{align*}
\Omega_k &= H^\top R^{-1} H + \Omega_{k|k-1}\\ 
&= \frac{1}{N} \sum^N_{i=1} (NH_i^{\top} R_i^{-1} H_i + \Omega_{k|k-1}).
\end{align*}
Define $\Omega_{i,k} := H^\top_i R^{-1}_i H_i + \frac{1}{N}\Omega_{i, k|k-1}$. Then, the updated information matrix of CKF can be obtained by solving the following distributed optimization problem
\begin{subequations}\label{eq:ECovProblem}
\begin{align}
\text{minimize}& \quad \sum_{i=1}^N (\zeta_i - \VEC{N \Omega_{i, k}} )^2 \label{eq:ECovObj}\\
\text{subject to}& \quad \zeta_1 = \cdots = \zeta_N\label{eq:DKFPECov_const}
\end{align}
\end{subequations}
where $\zeta_i \in \mathbb{R}^{n(n+1)/2}$ is the decision variable. Note that the minimizer $\zeta^* := [\zeta_i^{*};\dots;\zeta_N^{*}] \in \mathbb{R}^{Nn(n+1)/2}$ of the above optimization problem is nothing but the average of all $\text{vec}(N \Omega_{i, k})$, which corresponds to $\Omega_k$.

Define the Lagrangian for the problem (\ref{eq:ECovProblem}) as 
\begin{align}\label{eq:Lagrangian_ECov}
{\mathcal{L}}_{\Omega}(\zeta, \mu) &= \sum_{i=1}^N (\zeta_i - \VEC{N\Omega_{i,k}} )^2 + \mu^\top (L \otimes I) \zeta
\end{align}
where $\mu \in \mathbb{R}^{Nn(n+1)/2}$ is the dual variable. The saddle point equation for (\ref{eq:Lagrangian_ECov}) is given by
\begin{subequations}\label{eq:OptConditionECov}
\begin{align}
\begin{bmatrix}
-I & -L \otimes I\\
L \otimes I & 0
\end{bmatrix}
\begin{bmatrix}
\zeta^* \\ \mu^*
\end{bmatrix}
=
\begin{bmatrix}
-\bar{z}_{\Omega, k} \\ 0
\end{bmatrix}
\end{align}
\end{subequations}
where $\bar{z}_{\Omega, k} := [\VEC{N\Omega_{1, k}};\dots;\VEC{N \Omega_{N, k}}]$, and $\mu^*$ is the dual variable of the optimal point. From the similar arguments in the proof of {\em Lemma} \ref{Lemma:DKFConvergence}, we have
\begin{align*}
\zeta^* &= (1_{N} \otimes I) \frac{1}{N} \sum^N_{i=1} \VEC{N \Omega_{i, k}}\\
&= (1_{N} \otimes I) ( \VEC{H^\top R^{-1} H} + \frac{1}{N} \sum^N_{i=1} \VEC{\Omega_{i,k|k-1}}).
\end{align*}
This implies that the optimal solution $\zeta^*_i$ is the half vectorization of the average of the locally predicted information matrix corrected by the global information $H^\top R^{-1} H$.

Based on the above arguments, we propose a dual ascent type update rule for the error covariance matrix as
\begin{subequations}\label{eq:ECov_DA}
\begin{align}
\zeta_{i, l+1} &= \VEC{\Omega_{i, k}} - \sum_{j \in \mathcal{N}_i} a_{ij} (\mu_{i, l} - \mu_{j, l})\\
\mu_{i, l+1} &= \mu_{i, l} + \alpha_\mu \sum_{j \in \mathcal{N}_i} a_{ij} (\zeta_{i, l+1} - \zeta_{j, l+1}).
\end{align} 
\end{subequations}
where $\alpha_\mu$ is a step size such that $0 < \alpha_\mu < 2/\sigma_N^2$, which obtained by the similar arguments in the proof of {\em Lemma} \ref{lem:convergence}. 
Putting all pieces together, we propose a DKF algorithm  described in Algorithm \ref{Algo:DKF_DA}.

\begin{algorithm}
\caption{DA-DKF}\label{Algo:DKF_DA}
\begin{algorithmic}[1]
\State {\it{//Local prediction}}
\State \quad $\hat{x}_{i, k|k-1} = A_d \hat{x}_{i, k-1}$
\State \quad $P_{i, k|k-1} = A_d P_{i, k-1} A_d^\top + Q$, $\Omega_{i, k|k-1} = P^{-1}_{i, k|k-1}$
\State
\State {\it{//Distributed correction}} 
\State \quad $\lambda_{i, 0}, \mu_{i, 0} = 0$
\State \quad {\bf{while}} $l= 0,\dots, l^*-1$, {\bf{do}}
\State \quad \quad {\it{//Distributed estimate update}} (\ref{eq:DKF_DA})
\State \quad \quad $\xi_{i, l+1} = \hat{x}_{i, k|k-1} + K_{i, k} (y_{i, k} - H_i \hat{x}_{i, k|k-1}) - \psi_{i, l}$
\State \quad \quad $\lambda_{i, l+1} = \lambda_{i, l} + \alpha_\lambda \sum_{j \in \mathcal{N}_i} a_{ij} (\xi_{i, l+1} - \xi_{j, l+1})$
\State
\State \quad \quad {\it{//Distributed covariance matrix update}} (\ref{eq:ECov_DA})
\State \quad \quad $\zeta_{i, l+1} = \VEC{\Omega_{i, k}} - \sum_{j \in \mathcal{N}_i} a_{ij} (\mu_{i, l} - \mu_{j, l})$
\State \quad \quad $\mu_{i, l+1} = \mu_{i, l} + \alpha_\mu \sum_{j \in \mathcal{N}_i} a_{ij} (\zeta_{i, l+1} - \zeta_{j, l+1})$
\State \quad {\bf{end}}
\State \quad $\hat{x}_{i, k} = \xi_{i, l^*}$, $P_{i, k} = (\invVEC{\zeta_{i, l^*}})^{-1}$
\end{algorithmic}
\end{algorithm} 

In the structural point of view, the algorithm consists of {\it{local prediction}} step and {\it{distributed correction}} step as in CKF. In the local prediction step, each estimator locally predicts the estimate and the corresponding covariance matrix. In the distributed correction step, each estimator finds the optimal points for the state estimate  and its error covariance matrix, iteratively, by using the local measurement information and exchanging information with its neighbors. With sufficiently large $l^*$, locally updated $\xi_{i, l^*}$ and $P_{i, k}$ converge to those of CKF with tunable size of errors. 

\section{Numerical Experiments}\label{Sec:NE}
We have two examples for the developed theory. The first one is a simple academic example, while the second one is more practical one. 

\begin{example}\label{example}
Consider a system given by 
\begin{align*}
x_{k+1} &= \begin{bmatrix}  0.4& 0.9& 0& 0\\ -0.9& 0.4& 0& 0\\0& 0& 0.5& 0.8\\0& 0& -0.8& 0.5\\
		\end{bmatrix}x_k + w_k\\
y_k &= \begin{bmatrix}  1& 0& 0& 0\\ 1& 1& 0& 0\\0& 0& 1& 1\\0& 0& 1& 0\\
		\end{bmatrix} x_k + v_k
\end{align*} 
and $Q = 0.1$, $R = \diag(0.1, 0.2, 0.3, 0.1)$, and suppose that $4$ estimators are connected through a communication network whose Laplacian matrix is given by
\begin{align*}
L = \begin{bmatrix}
3& 0& -1& -2\\0&  2& -2&  0\\-1& -2&  4&  -1\\-2& 0& -1& 3
\end{bmatrix}.
\end{align*} 
The step sizes for the algorithm are chosen as   $\alpha_\lambda, \alpha_\mu = 0.01$.

Figure \ref{fig:ENorm_A1} shows that the average error norm defined by $\text{avg}(\| e_{i, k} \|) = \frac{1}{N} \sum_{i=1}^{N} \|\hat{x}_{i, k} - x_k\|$ decreases more rapidly as $l^*$ increases. Figure \ref{fig:PNorm_A1} also shows that as $l^*$ increases, the average error covariance norm defined by $\text{avg}(\| P_{i, k}\|) := \frac{1}{N} \sum_{i=1}^N \| P_{i, k} \|$ approaches $\| P_k \|$ which is  the norm of the error covariance  matrix of CKF. It is seen that, when $l^* = 50$, there is very little difference between $\text{avg}(\| P_{i, k}\|)$ and $\| P_k \|$ of CKF. \hfill $\square$
\end{example}

\begin{figure}[t]
\captionsetup[subfigure]{justification=centering}
    \begin{subfigure}[b]{0.24\textwidth}
        \includegraphics[scale=0.3]{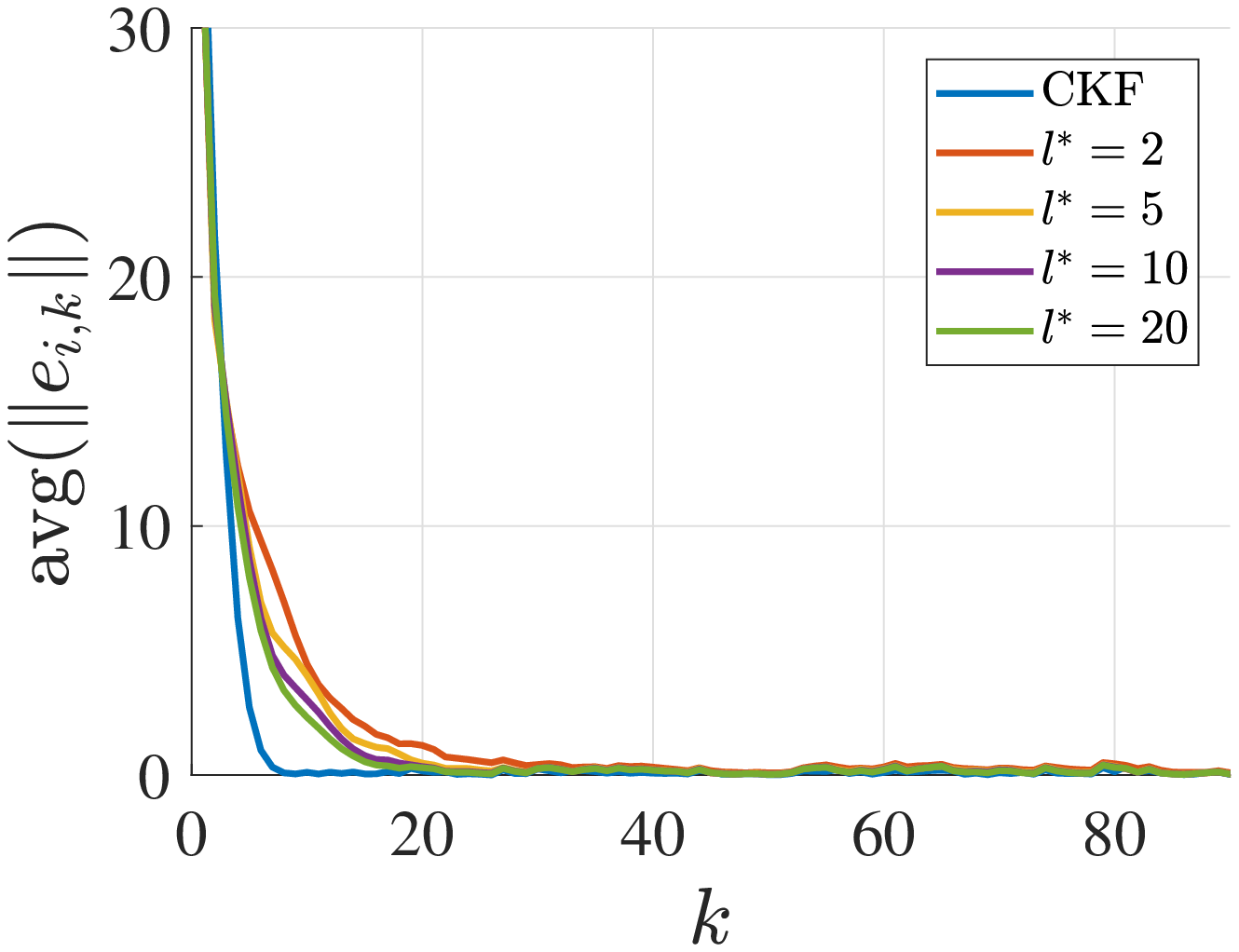}
        \caption{}
        \label{fig:ENorm_A1}
    \end{subfigure}
    \begin{subfigure}[b]{0.24\textwidth}
        \includegraphics[scale=0.3]{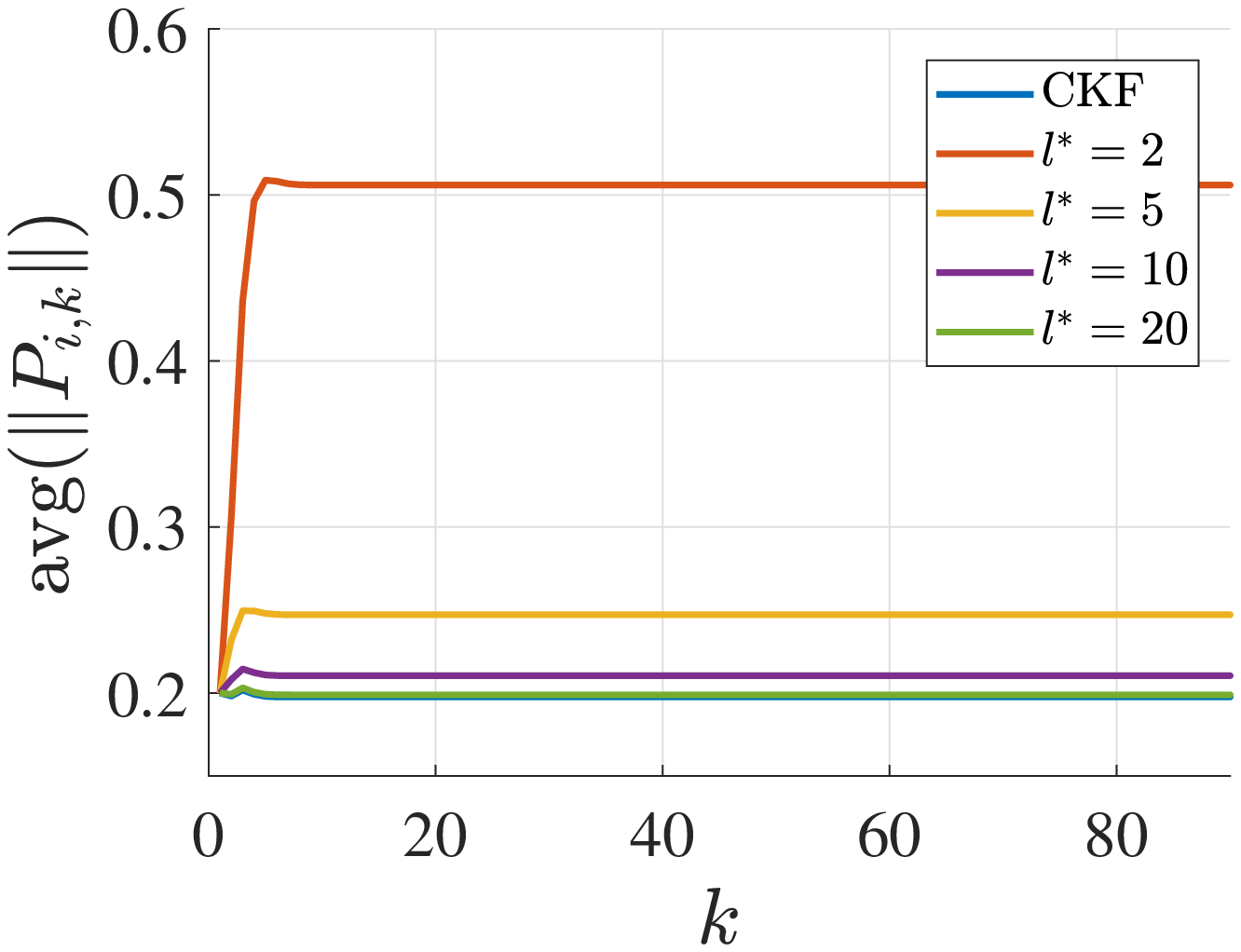}
        \caption{}
        \label{fig:PNorm_A1}
    \end{subfigure}
    \caption{The average of norm $\| e_{i, k} \|$ and the average of norms of local covariance matrix $\text{avg}(\| P_{i, k}\|)$ using DA-DKF.}
\end{figure}

\begin{example}
In this example, we evaluate DA-DKF with a network consisting of 50 estimators to estimate the state of a target system. The dynamics of the target system is described by
\begin{align*}
x_{k+1} = e^{A} x_k + w_k, \quad A = \begin{bmatrix} 0 & 0.5 & 0 & 0\\-0.5 & 0 & 0 & 0\\0 & 0 & 0 & -0.5\\0 & 0 & 0.5 & 0 \end{bmatrix}
\end{align*}
where $w_k \sim \mathsf{N}(0, Q)$ and $Q = 0.1$. The first and the third components of $x_k$ represent the $x$-axis position and $y$-axis position in the plane, respectively . 

The estimator $i$ knows $e^A$, $Q$, $H_i \in \mathbb{R}^{1 \times 4}$ and $R_i>0$, and each $H_i$ and $R_i$ is randomly chosen. 
The connections   among estimators are also randomly selected and the weight is $1$ when connected, and the maximum eigenvalue of $L$ is $18.5$. For all $i$, $P_{i,0} = Q$ and each component of the initial estimate $\hat{x}_{i, 0}$ is randomly chosen within $(-15, 15)$ as shown in Figure \ref{fig:k=0}. 
The parameters for DA-DKF were chosen as $\alpha, \beta = 10^{-5}$, $l^* = 10$. Figure \ref{fig:50DKF} shows four snapshots of the target system's position (black cross) and the each  estimator's estimate (red circles). The blue line is the trajectory of the target system. As time goes by (as $k$ increases), the estimates of the distributed Kalman-filters converge to the vicinity of the position of the target system. \hfill $\square$
\end{example}

\begin{figure}[t]
    \begin{subfigure}[t]{0.237\textwidth}
        \includegraphics[scale=0.31]{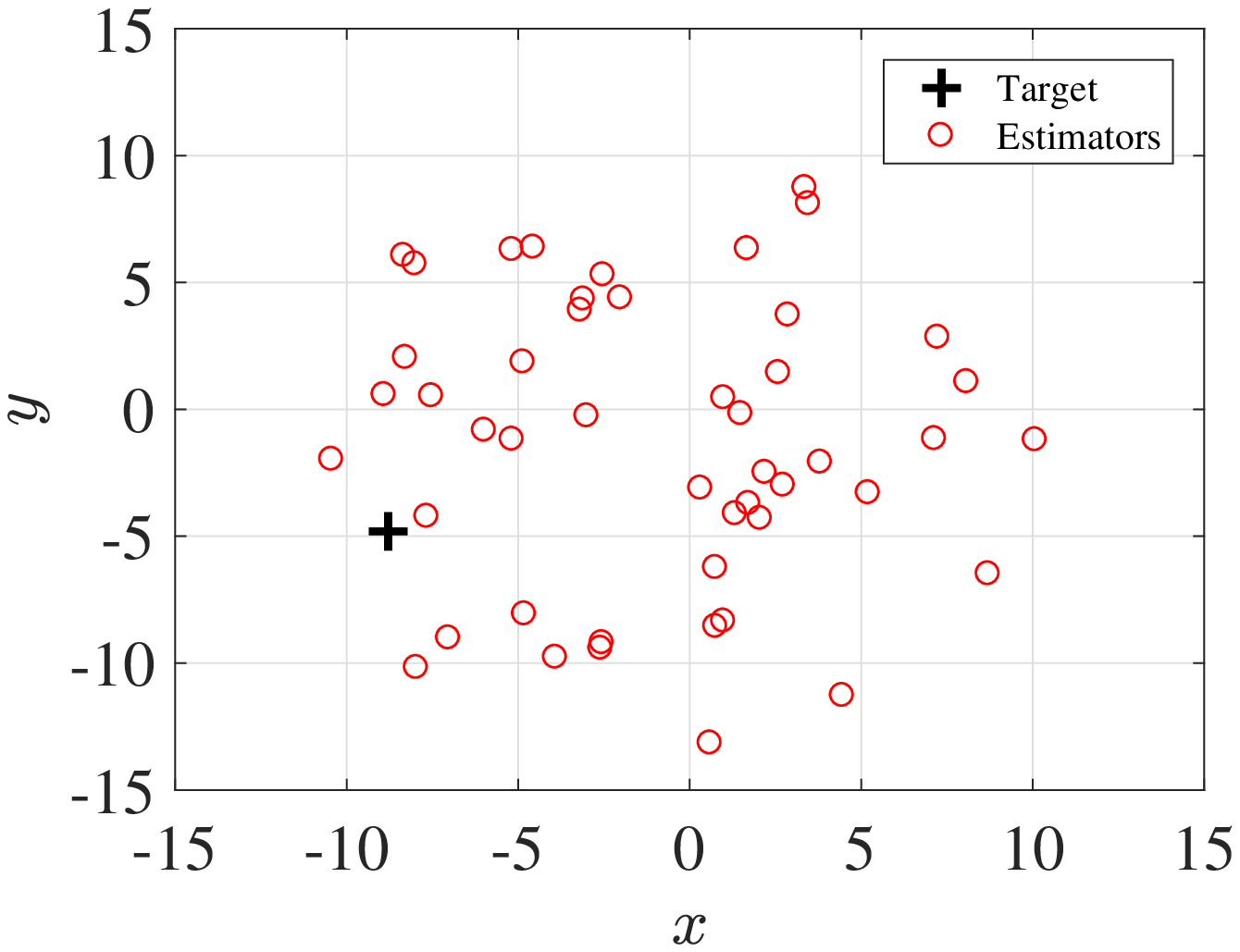}
		\caption{$k=0$}
        \label{fig:k=0}
    \end{subfigure}
    \begin{subfigure}[t]{0.237\textwidth}
        \includegraphics[scale=0.31]{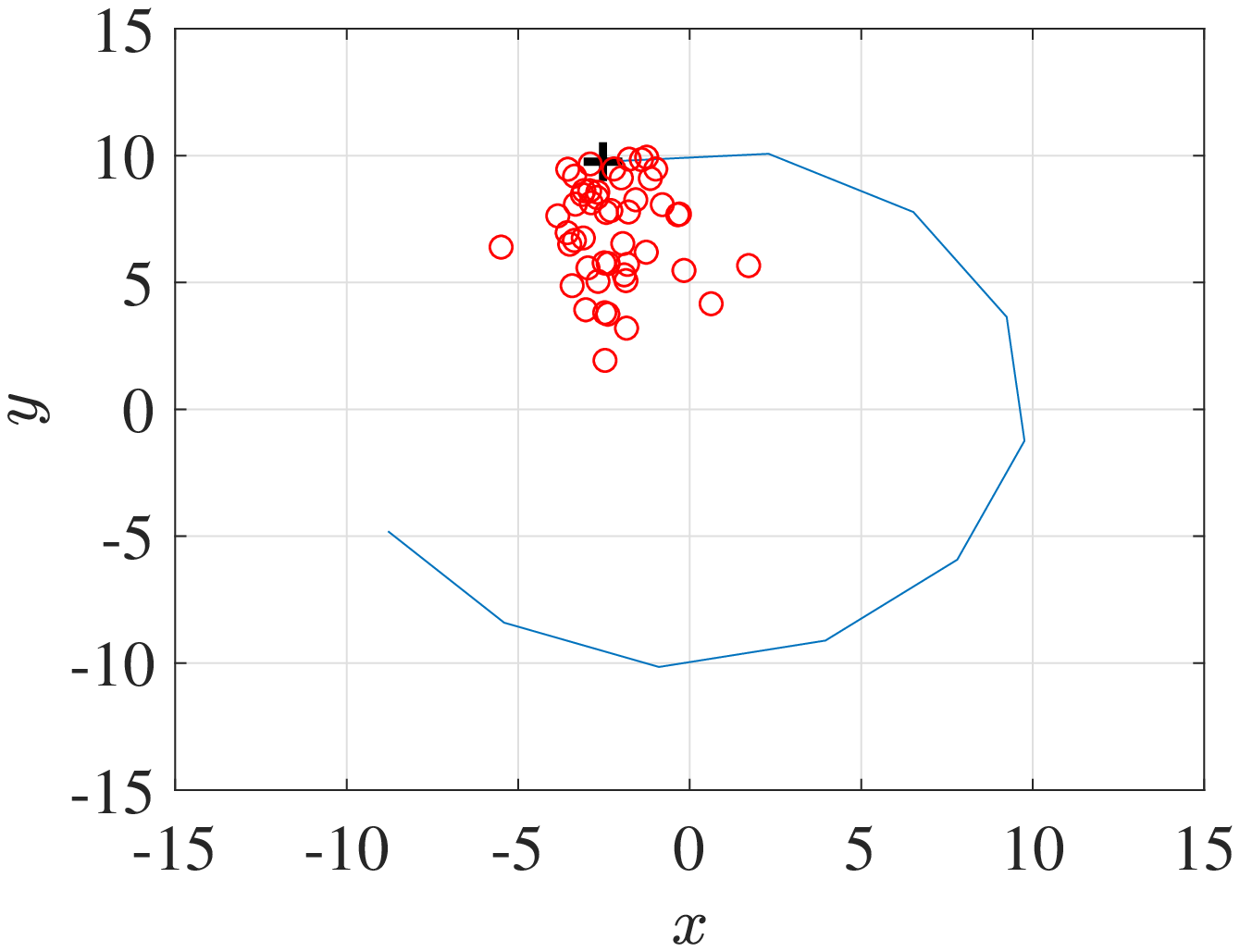}
        \caption{$k=10$}
        \label{fig:k=10}
    \end{subfigure}
    \begin{subfigure}[t]{0.238\textwidth}
        \includegraphics[scale=0.31]{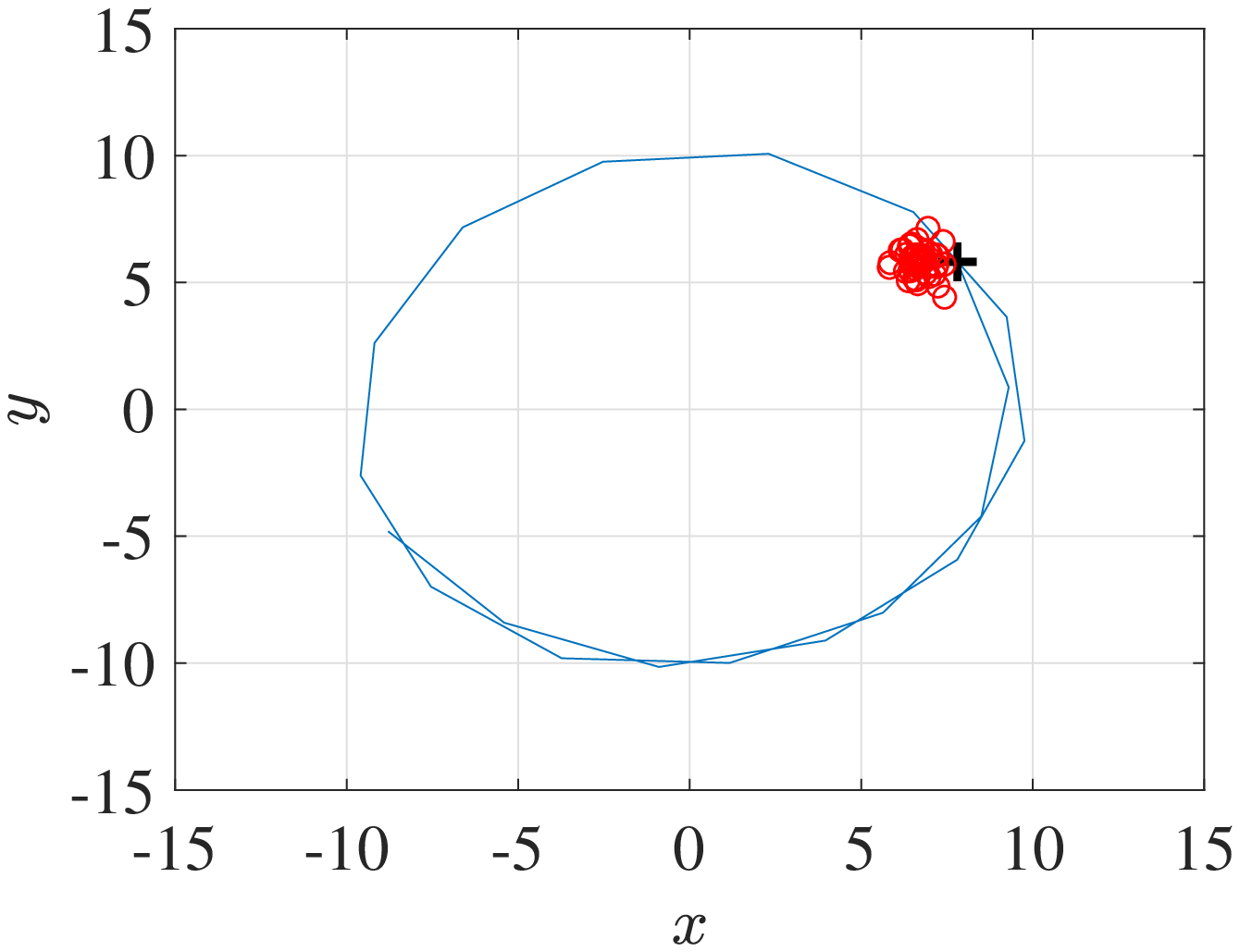}
        \centering \caption{$k=20$}
        \label{fig:k=20}
    \end{subfigure}
    \begin{subfigure}[t]{0.238\textwidth}
        \includegraphics[scale=0.31]{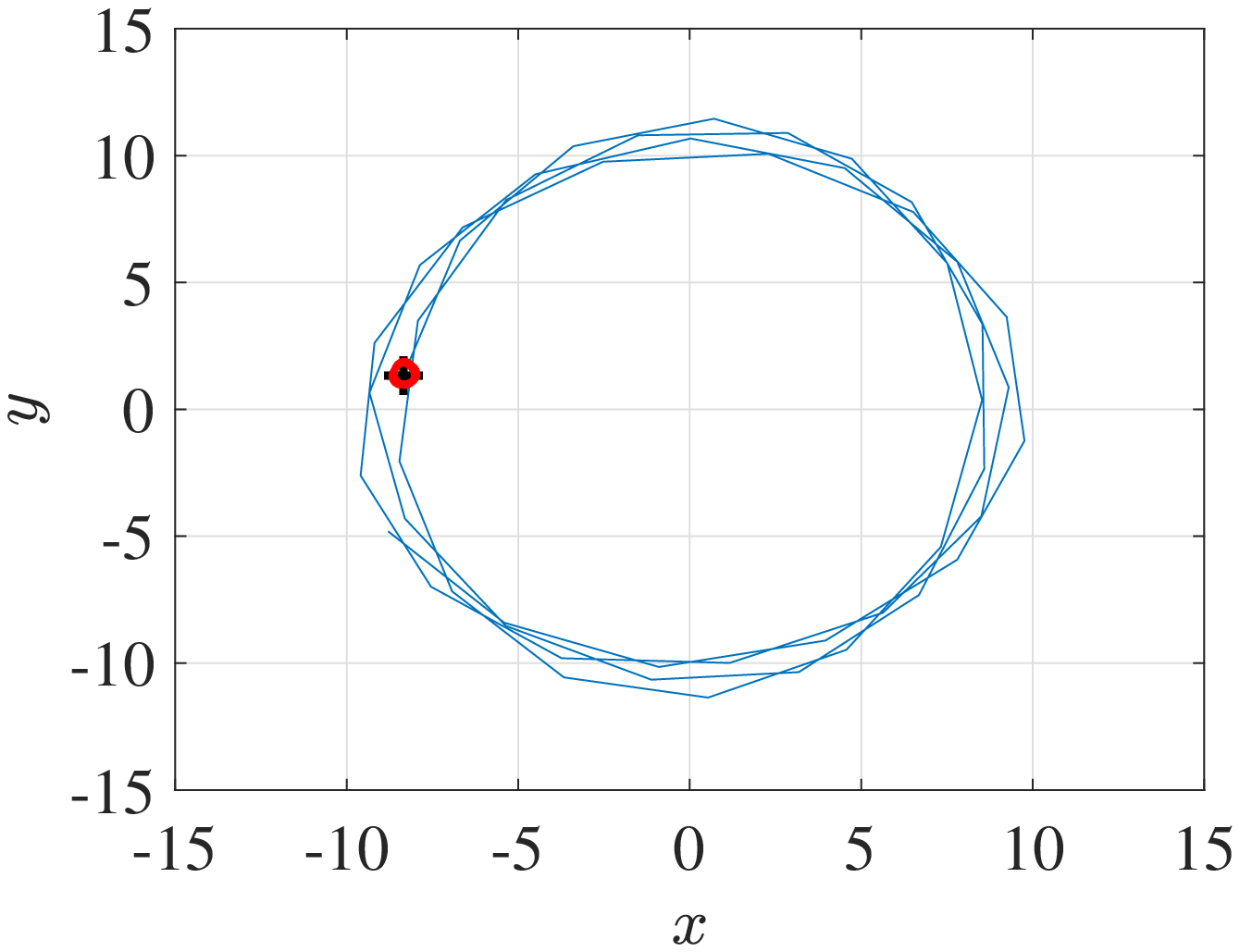}
 		\centering \caption{$k=50$}
        \label{fig:k=50}
    \end{subfigure}
    \caption{A sensor network with $50$ distributed Kalman-filters tracking a moving target using DA-DKF.}
\label{fig:50DKF}
\end{figure}

\section{Conclusions and Future Work}
This paper dealt with DKF from the optimization perspective. By observing that the correction step of Kalman-filtering is basically an optimization problem, we formulated DKF problem from the centralized one. The formulated problem is a quadratic consensus optimization problem. One of the recent DKF algorithms, Consensus on Information \cite{Battistelli2015TAC} was reinterpreted from the distributed optimization perspective. In addition, various DKF algorithms can be derived, by employing many existing distributed optimization methods to DKF problem. As an instance, DA-DKF has been presented, employing the distributed dual ascent method, and the algorithm has been validated with numerical experiments.

For the future work, we plan to analyze the effect of the residuals of the previous iteration $k$, especially how the residuals affect the convergence. In addition, researches on more practical obstacles, such as considering the time-varying network topology, reducing communication loads, will be conducted.

\bibliographystyle{IEEEtran}
\bibliography{mybib}



\end{document}